\documentclass[11pt]{article}
\usepackage{amsfonts}
\usepackage[mathscr]{eucal}
\usepackage{mathrsfs}
\usepackage{amsmath}
\usepackage{amssymb}

\newtheorem{theorem}{Theorem}[section]
\newtheorem{lemma}[theorem]{Lemma}
\newtheorem{proposition}[theorem]{Proposition}

\newtheorem{remark}[theorem]{Remark}
\newtheorem{corollary}[theorem]{Corollary}

\newcommand{\al}{{\alpha}}

\newcommand{\eps}{\varepsilon}
\newcommand{\la}{\lambda}

\newcommand{\Ac}{\mathfrak{A}}

\newcommand{\sB}{\mathscr{B}}

\newcommand{\sD}{\mathscr{D}}

\newcommand{\p}{{\mathbb P}}
\newcommand{\R}{{\mathbb R}}
\newcommand{\T}{{\mathbb T}}

\newcommand{\cA}{{\cal A}}

\newcommand{\cO}{{\cal O}}

\newcommand{\cD}{{\cal D}}

\newcommand{\hf}{{\frac12}}

\newcommand{\g}{{\nabla}}

\newcommand{\pd}{\partial}

\newcommand{\di}{{\rm div\, }}
\newcommand{\wrt}{{with respect to }}

\newenvironment{declaration}[1]{\trivlist
\item[\hskip \labelsep{\bf #1 }]\ignorespaces}{\endtrivlist}
\newenvironment{proofof}[1]{\begin{declaration}{#1}}{\hfill
$\square$ \end{declaration}}
\newenvironment{proof}{\begin{proofof}{Proof.}}{\end{proofof}}

\begin{document}
\title{A squeezing property and its applications to \\ a description of long time
behaviour in \\  the 3D viscous primitive equations
}
\author{Igor Chueshov\footnote{\small e-mail:
 chueshov@karazin.ua}
  \\
 \\Department of Mechanics and Mathematics, \\
Karazin Kharkov National University, \\ Kharkov, 61022,  Ukraine\\  }

\maketitle

\begin{abstract}
We consider the 3D viscous primitive equations
with periodic boun\-dary conditions. These equations arise in the study of ocean dynamics
and generate a dynamical system in a Sobolev $H^1$ type space.
Our main result
establishes the so-called squeezing property in the Ladyzhenskaya form
for this system.
As a consequence of this property we prove (i) the finiteness of the fractal dimension
of the corresponding  global attractor, (ii) the existence of finite number of
determining modes, and (iii)
 ergodicity of a related random kick
model. All these results provide a new information
concerning long time dynamics  of oceanic motions.
\end{abstract}

\section{Introduction}
We deal with the  3D viscous primitive equations
which arise in geophysical fluid dynamics  for modeling  large scale phenomena
in oceanic motions.
These equations are based on the so-called hydrostatic approximation of the 3D
 Navier-Stokes equations for velocity field $u$ which also contain a rotational (Coriolis) force and
coupled to thermo- and salinity diffusion-transport equations
(see, e.g., the survey~\cite{temam2008-srw} and the references therein).
In this model a small variation of
the density  $\rho$ of the fluid
is taken into account via the buoyancy term only and
has the form
$\rho=\rho_0-\alpha T+\beta s$,
 where  $T$ is the temperature and $s$ denotes
the salinity.
If we take into account one of the factors ($T$ or $s$) or the diffusivities for $T$ and $s$
are equal, then their effect in the dynamics can be represented simply by
the density $\rho$, or equivalently the buoyancy $b=\rho_0-\rho =\alpha T-\beta s$
(see~\cite{temam2008-srw}). This is why we prefer to deal with the variables $u$ and $b$.
In order to simplify our mathematical presentation we impose on $u$ and $b$ the
periodic boundary conditions of the same type as in \cite{petcu} (see also
~\cite{temam2008-srw} and the references therein). Thus we consider
our problem in the domain
 \[
 \cO= (0,L_1)\times (0,L_2)\times (-L_3/2,L_3/2) \subset \R^3
 \]
 and denote the spatial variable in $\cO$ by $\bar{x}=(x,z)=(x_1,x_2,z)\in \cO$.
We  use the notation $\g$, $\di$ and $\Delta$ for the gradient, divergence and  Laplace operators
in two dimensional variable $x=(x_1,x_2)$.
Below we also denote by $\Delta_{x,z}$  the 3D Laplace operator.
The notations $\g_{x,z}$ and $\di_{x,z}$ have a
similar meaning.
\par
We consider
 the following  equations
for the fluid velocity field
\[
u=(v(\bar{x},t);w(\bar{x},t))=(v^1(\bar{x},t);v^2(\bar{x},t);w(\bar{x},t)),~~ \bar{x}=(x,z),
\]
 for the pressure $P=P(\bar x,t)$ and for the buoyancy $b= b(\bar{x},t)$:
\begin{equation}\label{fl.1}
   v_t +(v,\g)v+ w \pd_z v-\nu\Delta v -\nu\pd_{zz} v +f v^{\perp}+\nabla P=G_f\quad {\rm in\quad} \cO
   \times(0,+\infty),
\end{equation}
   \begin{equation}\label{fl.2}
   \di v +\pd_z w=0 \quad {\rm in}\quad \cO
   \times(0,+\infty),
  \end{equation}
   \begin{equation}\label{fl.3}
   \pd_z P=b \quad {\rm in}\quad \cO
   \times(0,+\infty),
  \end{equation}
    \begin{equation}\label{fl.4}
   b_t +(v,\g) b+ w \pd_zb-\nu\Delta b -\nu\pd_{zz} b =G_b\quad {\rm in\quad} \cO
   \times(0,+\infty),
  \end{equation}
 where $\nu>0$ is the dynamical viscosity, $G_f$ and $G_b$ are volume sources,
 and $f$ is the Coriolis parameter.
We denote  $v^{\perp}=(-v^2;v^1)$.
We supplement (\ref{fl.1})--(\ref{fl.4}) with  the periodic  boundary
   conditions imposed  on  $(v;w)$, $P$ and $b$.
   As in   \cite{petcu,temam2008-srw}
   we also assume that
 $v$ and $P$ are even with respect to $z$ and $w$ and $b$ are odd  with respect to $z$
 (hence $w\big|_{z=0}=0$ and $b\big|_{z=0}=0$).
These requirements and also (\ref{fl.2}), (\ref{fl.3}) lead to the following relations
\begin{equation}\label{rep-w}
w(\bar x,t)=-\int_0^{z} \di v(x,\xi,t) d\xi,~~ P(\bar x,t)=p(x,t)+\int_0^{z} b(x,\xi,t) d\xi
\end{equation}
for every $\bar{x}=(x,z)\in \cO$.
It is important (see \cite{CaoTiti}
for the basic discussion) that the pressure $p$ is independent of the vertical variable $z$ and depends
on 2D (horizontal) variable $x$ only. Exploring intensively this observation the authors of the paper
 \cite{CaoTiti} have implemented a new effective approach for proving of the global well-posedness
 for  problems like  (\ref{fl.1})--(\ref{fl.4}), see the discussion below.
\par
The  system in (\ref{fl.1})--(\ref{fl.4}) was
 studied by many authors for the different types of boundary conditions
 (see  the survey \cite{temam2008-srw} and the references
therein). The existence  of week solutions \cite{LTW-ocean}
  and (global) well-posedness of strong solutions were
established (see \cite{CaoTiti} and also \cite{kobelkov,kuka,petcu}).
The existence of a global attractor for the viscous 3D primitive equations
was proved in \cite{ju} (see also the paper \cite{petcu} devoted
to the periodic case).
However,  to the  best of our knowledge,  the question concerning dimension of this attractor
is still open.
We note that the author of \cite{ju} claimed the finiteness of the dimension  in 2007
with the reference to a forthcoming paper which is not published yet.
\par
It should be
also noted that stochastic perturbations of (\ref{fl.1})--(\ref{fl.4})
were studied in \cite{DGTZ-stoch,GH-stoch,GS-stoch} (see also the references in these publications)
and the paper \cite{EG-kick} deals with a random kick forcing of the primitive equations.
\par

Our goal is to show that the system in (\ref{fl.1})--(\ref{fl.4})
possesses an additional regularity and satisfies a squeezing property in the Ladyzhenskaya
form (see \cite{lad1,lad2}).
This property implies several important facts about dynamics of the system and,
in particular, the finiteness of the fractal dimension of the attractor (see Corollary~\ref{co:dim-attr})
and the existence of finite number of determining modes or functionals (Corollary~\ref{co:det}).
The squeezing property also plays an important role in the study of turbulent behaviour
in some classes of dynamical systems with random kicks arising in the fluid dynamics
(see, e.g., \cite{KS-00,KS-book}). In this paper we apply  the squeezing property established
 and also the theory developed in  \cite{KS-book}
to present a result on ergodicity of (\ref{fl.1})--(\ref{fl.4}) forced by random bounded kicks
with appropriately chosen frequency  (see Corollary~\ref{co:erg}).
\par
As we will  see below the
 main ingredient of our argument is the uniform smoothing
 stated in  Theorem~\ref{th-prem1}.
There are two ways to obtain this asymptotic regularity property.
One method involves appropriate (spatial) multipliers
which produce higher kinetic energy type norms in the corresponding
energy balance relation.
Another method consists of two steps. In the first step the standard phase space
estimates for higher time derivatives  of solutions are established. Then, on the second step,
spatial smoothness of solutions is derived  via appropriate elliptic regularity
of the corresponding stationary problem. Both methods requires some compatibility
conditions involving the initial data and the nonlinearity
(see, e.g., \cite{temam-jde82} for a general presentation).
In the periodic case  these compatibility relations are satisfied automatically.
However the first method is a one-step method and looks simpler.
Moreover, it requires minimal hypotheses concerning
time regularity of external forces.
This why it was used in \cite{petcu,temam2008-srw}
and also in this paper. In the case of other boundary conditions (free type as in \cite{CaoTiti}
or mixed free-Dirichlet as in \cite{kuka})
we need rather careful analysis of compatibility
conditions and elliptic properties of the  stationary problem.
 We plan to analyze this issue in our further publications.
\par

It is commonly recognized  (see \cite{CaoTiti,kobelkov,kuka,petcu,temam2008-srw})
that the main difficulty arising in the study of the primitive equations in (\ref{fl.1})--(\ref{fl.4})
is related to the hydrodynamical part. The calculations which should be made to obtain
appropriate bounds for a buoyancy type variable are either
standard or repeat the corresponding argument for the fluid
variables in the simplified form.
This is why below we state our results for the full problem (\ref{fl.1})--(\ref{fl.4}),
but in the proof we concentrate on the fluid equations only.
Moreover, for the transparency of our argument  we even assume that
the  buoyancy equation is absent. This corresponds to the case of
the zero  buoyancy ($b(\bar{x},t)\equiv 0$) which is possible
when $G_b\equiv 0$ and $b(\bar{x},t=0)=0$.
\smallskip
\par
The paper is organized as follows.
\par
Section \ref{sect2} contains some preliminary material. Here we represent the primitive equations
in some standard form, introduce  Sobolev type spaces and
state  the well-known result on  well-posedness of  strong
solutions.
In Section \ref{sect:res}  we state our main results. The key
 outcome  of Theorem \ref{th-prem1}      is the existence of an absorbing ball
 lying in an $H^2$ type space. This allows us not only guarantee the existence of a global
 attractor but also
provides an important step in the proof of the squeezing in Theorem \ref{th-main1}.
In this section we also apply our main results to prove  the finiteness of
the dimension of the global attractor,
the existence of finite number of determining functionals, and ergodicity of a random kicks forced model
generated by (\ref{fl.1})--(\ref{fl.4}).
In Section \ref{sect:prf} we demonstrate the main steps of the proofs
of Theorems \ref{th-prem1} and \ref{th-main1}
in the case when  the buoyancy variable $b$ is absent.

\section{Preliminaries}\label{sect2}

In this section we rewrite
system  (\ref{fl.1})--(\ref{fl.4}) in the canonical form, introduce  appropriate Sobolev
type spaces and  quote the result on the strong well-posedness.
\par
Using the representations in (\ref{rep-w})  we arrive at the system
 \begin{align}\label{fl.1a}
   v_t +(v,\g)v & -\left[\int_0^{z} \di vd\xi \right]\pd_z v -\nu [\Delta v +\pd_{zz} v]
 +f v^{\perp}  \notag
  \\ & =-\nabla \left[ p(x,t)+\int_0^{z} b d\xi\right]
 +G_f\quad {\rm in\quad} \cO
   \times(0,+\infty),
\end{align}
    \begin{equation}\label{fl.4a}
   b_t +(v,\g) b - \left[\int_0^{z} \di vd\xi \right] \pd_zb-\nu [\Delta b +\pd_{zz} b] =G_b\quad {\rm in\quad} \cO
   \times(0,+\infty),
  \end{equation}
  supplied with the conditions:
   \begin{equation}\label{fl.4a-bc1}
 {\rm div}\,\int_{-L_3/2}^{L_3/2} v dz=0;~~
v ~\mbox{is periodic in $\bar{x}$ and even in $z$},~~\int_{\cO} v d\bar{x}=0;
  \end{equation}
and
   \begin{equation}\label{fl.4a-bc2}
  b ~\mbox{is periodic in $\bar{x}$ and odd in $z$}.
   \end{equation}
 We also need to add initial data for $v$ and $b$:
   \begin{equation}\label{fl.4a-id}
 v(0)=v_0,~~~b(0)=b_0.
  \end{equation}
  \par
We denote  by $\dot{H}_{per}^s(\cO)$ the Sobolev space of order $s$
consisting of periodic functions  such that $\int_{\cO} f dx=0$.
To describe fluid velocity fields
we introduce the following spaces with $s=0,1,2$:
\par
\begin{equation*}
V_s=\left\{
v=(v^1;v^2)\in [\dot{H}_{per}^s(\cO)]^2 :\; v^i~\mbox{is even in}~z,\;
 {\rm div}\,\int_{-L_3/2}^{L_3/2} v dz=0
  \right\}.
\end{equation*}
We also denote $H=V_0$ and
 equip   $H$ with $L_2$ type norm $\|\cdot\|$
and denote by $(\cdot,\cdot)$ the corresponding inner product.
The spaces $V_1$ and $V_2$ are endowed  with the norms  $\|\cdot\|_{V_1}= \|\nabla_{x,z}\cdot\|$
and $\|\cdot\|_{V_2}= \|\Delta_{x,z}\cdot\|$.
\par
To describe the buoyancy we need the spaces
\begin{equation*}
E_s=\left\{
b\in \dot{H}_{per}^s(\cO) :\; b~\mbox{is odd in}~z
  \right\},~~~s=0,1,2,
\end{equation*}
equipped with the standard norms. We also denote $W_s=V_s\times E_s$
with the standard  (Euclidean) product norms.
\par
As it was already mentioned, starting with \cite{CaoTiti} the global well-posedness of
the equations  in  \eqref{fl.1a} and \eqref{fl.4a} was studied by many authors~\cite{kobelkov,kuka,petcu,temam2008-srw}.
In this section we quote the result from
 \cite{petcu} (see also \cite{temam2008-srw}) on well-posedness
 of strong solutions under the periodic boundary conditions.

\begin{proposition}[\cite{petcu,temam2008-srw}]\label{pr:titi}
Let $G_f\in V_0$ and $G_b\in E_0$. Then for every $U_0=(v_0;b_0)\in W_1$ problem (\ref{fl.1a})--(\ref{fl.4a-id})
 has a unique strong solution $(v(t);b(t))$:
\begin{equation*}
U(t;U_0)\equiv (v(t);b(t))\in C(\R_+; W_1)\cap L_2(0,T; W_2),~~~\forall T>0.
\end{equation*}
Moreover, this solution continuously depends on time $t$ and initial data $U_0$
and generates a dynamical system $(S_t, W_1)$ with
the phase space $W_1$ and the evolution operator $S_t$
defined by the relation $S_tU_0=U(t;U_0)$.This evolution operator $S_t$
possesses the following  Lipschitz property
\[
\|S_tU-S_tU_*\|_{W_1}\le C_{T,R}\|U-U_*\|_{W_1},~~ t\in [0,T],
\]
for every $T>0$ and $U,U_*\in B(R)\equiv\{U : \|U\|_{W_1}\le R\}$.
\end{proposition}
\begin{remark}
{\rm
Formally \cite{petcu,temam2008-srw} do not contain the proofs of continuity of
$S_tU$ \wrt $t$ and the Lipschitz property  in $W_1$. However almost the
same calculations as in \cite{ju}
allow us to show these properties.
We also note  that in the case considered
 unique global solvability can be also established
in  smoother classes. For instance, if
 $G_f\in V_{m-1}$, $G_b\in E_{m-1}$ and $U_0=(v_0;b_0)\in W_m$
for some $m\ge 2$ then (see \cite{petcu,temam2008-srw}) the solution $U(t)$
belongs to the class
$C(\R_+; W_m)\cap L_2(0,T; W_{m+1})$ for every $T>0$.
This observation is important in our further considerations due to possibility to
use smooth approximations of solutions in our calculations.
 }
\end{remark}
We conclude this preliminary section with
 the statement of the following well-known uniform Gronwall lemma.
\begin{lemma}[\cite{Temam}] \label{gronwall}
Let $g$, $h$, $y$ be nonnegative locally integrable functions on $[t_0,+\infty)$
such that
\[
\frac{dy}{dt}\le gy+h,
\]
and
\[
\int_t^{t+1} g(s)ds\le a_1,~~\int_t^{t+1} h(s)ds\le a_2,~~\int_t^{t+1} y(s)ds\le a_3
\]
for any $t\ge t_0$, where $a_i>0$  are constant. Then
\[
y(t+1)\le \left( a_3+a_2\right)e^{a_1}~~~\mbox{for any $t\ge t_0$.}
\]
\end{lemma}

\section{Results}\label{sect:res}
We start with the following assertion which partially follows
from Proposition~\ref{pr:titi} and is based on the calculations
given in \cite{petcu}.

\begin{theorem}[Smooth Absorbing Ball]\label{th-prem1}
Let $G_f\in V_1$, $G_b\in E_1$ and also $(\pd_zG_f;\pd_zG_b)\in \big[L_6(\cO)\big]^3$.
Then the dynamical
system $(S_t, W_1)$ generated by problem (\ref{fl.1a})--(\ref{fl.4a-id}) is
dissipative and compact\footnote{For the definitions see \cite{BabinVishik}, for instance.}. Moreover,
there exists $R_*>0$ such that the ball $$
\sB=\{ U\in W_2: \|U\|_{W_2}\le R_*\}
$$ is absorbing, i.e.,
for any bounded set $B$ in $W_1$ there exists $t_B$ such that
\[
S_tB\subset \sB ~~~\mbox{for all}~ t\ge t_B.
\]
In this ball $\sB$ there is a forward invariant  absorbing set $\sD$.
\end{theorem}
We note that the dissipativity and  compactness  of $(S_t, W_1)$  stated in Theorem~\ref{th-prem1}
is known  from \cite{petcu} under slightly weaker conditions on
$G_f$ and $G_b$. However the existence of an absorbing compact set
which is {\em bounded  in} $W_2$
requires additional hypotheses and additional calculations.
\par
\begin{remark}\label{re:zero-force}
{\rm
In the case when $G_f\equiv 0$ and $G_b\equiv 0$
one can prove (see some argument below in Section~\ref{sect3.1}) that
\[
\forall\eps>0,~~\exists \ T(R,\eps):~~ \|S_t U\|_{W_1}\le\eps
~~~\forall t\ge T(R,\eps),~ \|U\|_{W_1}\le R.
\]
This means that in the case of vanishing sources  the zero equilibrium state
is asymptotically stable.
}
\end{remark}

Using the standard results (see, e.g., one of the monographs
\cite{BabinVishik,Chueshov,Temam}) on  the existence of global attractors we can
derive from Theorem~\ref{th-prem1} the following assertion.
\begin{corollary}[Global Attractor]\label{co:exist}
Let the hypotheses of Theorem~\ref{th-prem1} be in force.
Then the dynamical
system $(S_t, W_1)$ generated by problem (\ref{fl.1a})--(\ref{fl.4a-id})
possesses a compact global attractor $\Ac$ which is a bounded
set in $W_2$. Moreover, in the case when $G_b=0$ the subspace $V_1\times \{0\}\subset W_1$
attracts the trajectories with exponential speed and the attractor  $\Ac$
lies in $V_2\times \{0\}$.
\end{corollary}

\begin{remark}\label{re:sm-atrr}
{\rm
The result on the existence of a global attractor for 3D viscous primitive equations
was known before, see \cite{ju}, and also Remark 2.1 in \cite{petcu} for the periodic
case. Our  improvement is that we state the boundedness of $\Ac$  in  an $H^2$ type Sobolev space.
One can also show that in the case of sufficiently smooth sources $G_f$
and $G_b$ the attractor $\Ac$ possesses additional spatial smoothness.
Using the method presented in \cite{petcu} we can even prove Gevrey regularity of the elements from the
attractor provided the sources possess appropriate smoothness properties.
However we do not pursue these improvements because our main goal is a squeezing property
with application to  dimension and ergodicity.
}
\end{remark}

In the space $W_1$ we consider the bilinear form
\[
a(U,U_*)=\int_\cO \left[ \g_{x,z} v\cdot\g_{x,z} v_* + \g_{x,z} b\cdot\g_{x,z} b_*\right] dxdz,
\]
where $U=(v;b)$ and $U_*=(v_*;b_*)$ are elements from $W_1$.
This form generates a positive self-adjoint operator $A$ with
a discrete spectrum. Due to periodicity of the problem
the corresponding eigenfunctions and eigenvalues can be easily described.
We assume that they are reorganized in such way that
\[
A e_k=\la_k e_k,~~~0<\la_1\le \la_2\le\ldots,~~~\lim_{k\to+\infty}\la_k=+\infty,
\]
and,
moreover, $\{ e_k\}$ is the orthonormal basis in $W_0$.
We denote by $P_N$ the orthoprojector onto  Span$\{e_1,\ldots,e_N\}$
and $Q_N=I-P_N$.
\par
Now we can state our main result.
\begin{theorem}[Squeezing]\label{th-main1}
Let the hypotheses of Theorem~\ref{th-prem1} be in force.
 Then the dynamical
system $(S_t, W_1)$ generated by problem (\ref{fl.1a})--(\ref{fl.4a-id})
on the forward invariant  absorbing set $\sD$
possesses the properties:
\begin{itemize}
\item {\bf Lipschitz property:}
For any $U,U_*\in\sD$ we have
the relation
\begin{equation}\label{lip-s-d}
\|S_tU-S_tU_*\|_{W_1}\le C_{\sD}e^{\alpha_{\sD}t}\|U-U_*\|_{W_1},~~~\forall\, t\ge 0,
\end{equation}
where $C_{\sD}$ and $\alpha_{\sD}$ are positive constants.
\item
{\bf Squeezing property:}
For every $T>0$ and $0<q<1$ there exists $N=N(T,q)$ such that
\begin{equation}\label{sq-ineq}
\|Q_N[S_T U-S_TU_*]\|_{W_1}\le q \|U-U_*\|_{W_1}
\end{equation}
for any $U$ and $U_*$ from the absorbing forward invariant set $\sD$.
\end{itemize}
\end{theorem}
The squeezing property stated in Theorem~\ref{th-main1} allow us to apply
Ladyzhenskaya Theorem on the dimension of invariant sets (see \cite{lad1,lad2}
and also Theorem 8.1 in \cite[Chap.1]{Chueshov}). This yields the following
assertion.
\begin{corollary}[Finite Dimension]\label{co:dim-attr}
Under the hypotheses of Theorem~\ref{th-prem1} the global attractor $\Ac$
of  the dynamical
system $(S_t, W_1)$ generated by problem (\ref{fl.1a})--(\ref{fl.4a-id})
has a finite fractal dimension.
\end{corollary}
This means that  the asymptotic dynamics in this model
is finite-dimensional  and topologically can be represented by a compact set in $\R^d$
with appropriate $d<\infty$.
\smallskip\par
The next outcome of  the squeezing property is the existence of a
 finite number of  (asymptotically) determining modes.
We note that this notion was introduced in \cite{FP-det}
for 2D Navier-Stokes system and was studied by many authors
(see the discussion in the recent paper \cite{FJKT-det}).
 For a general theory  of the determining functionals we refer to \cite{CT-det},
see also \cite{Chueshov} and \cite{cl-mem} for a development of this theory
based on the notion of the completeness defect.
\begin{corollary}[Determining Modes]\label{co:det}
Let  the hypotheses of Theorem~\ref{th-prem1} be in force and $N$ be such that \eqref{sq-ineq}
holds for some $T>0$ and $0<q<1$. Then the dynamical system   $(S_t, W_1)$ possesses $N$ determining modes.
This means that the relation
\begin{equation}\label{N-cnv}
\lim_{t\to+\infty}\|P_N[S_t U-S_tU_*]\|_{W_1}=0~~\mbox{for some  $U, U_*\in W_1$}
\end{equation}
implies that $\|S_t U-S_tU_*\|_{W_1}\to 0$ as $t\to+\infty$.
\end{corollary}
\begin{proof}
 The relation in (\ref{sq-ineq}) implies
that
\begin{equation}\label{sq-ineq2}
\|S_T U-S_TU_*\|_{W_1}\le q \|U-U_*\|_{W_1} + \|P_N[S_T U-S_TU_*]\|_{W_1}, ~~~0<q<1,
\end{equation}
for any $U$ and $U_*$ from the absorbing forward invariant set $\sD$.
Iterating  (\ref{sq-ineq2}) we obtain that
\begin{equation}\label{sq-ineq3}
\|S_T^n U-S_T^nU_*\|_{W_1}\le q^n \|U-U_*\|_{W_1} +
\sum_{k=1}^{n} q^{n-k}\|P_N[S_T^{k} U-S_T^{k}U_*]\|_{W_1}
\end{equation}
for every $n=1,2,\ldots$, where $U$ and $U_*$ are from  $\sD$.
This inequality and also the Lipschitz property
\eqref{lip-s-d} makes it possible to prove
that  relation \eqref{N-cnv}
implies that $\|S_t U-S_tU_*\|_{W_1}\to 0$ as $t\to+\infty$.
For some details in the abstract situation we refer to the proof of
Theorem~1.3 in \cite[Chap.5]{Chueshov}.
\end{proof}

\begin{remark}
{\rm
Using the same idea as in \cite{cl-mem} it is also possible to
derive from (\ref{sq-ineq3}) the existence of other finite families
of determining functionals.
We also point out the recent paper \cite{petcu2012} which contains the proof
of the existence  of a finite number of determining modes (and also nodes and local volume
averages) by another method.
}
\end{remark}
\par
Now we consider  an  example of an application of the results above to ergodicity of a
random kicks forced model generated by problem (\ref{fl.1a})--(\ref{fl.4a-id}).
We deal with the simplest situation and suppose that
 $G_f\equiv 0$ and $G_b\equiv 0$.
Moreover, we consider a model in some fixed ball from $W_1$
and assume that the frequency of kicks is smaller than a certain  critical value (depending  on the radius of the ball
and the amplitude of the kicks).
We plan to remove all these restrictions and discuss further developments
 in our joint studies with S.Kuksin and A.Shirikyan.

\par
Let $\{\eta_k\}$ be i.i.d.\ random variables in $W_1$.
Assume that the support of the distribution $\cD(\eta_1)$ in $W_1$  contains the origin and is bounded,
i.e., there exists $ R_{kick}>0$ such that
\[
{\rm    supp}\,\cD(\eta_1)\subset B(R_{kick})\equiv\left\{ U\in W_1 :~\| U\|_{W_1}\le R_{kick}\right\}.
\]
Here and below we keep the notation $B(R)$ for the ball with the center at $0$ of the radius $R$
in the space $W_1$.
\par
Let us fix $\widehat{R}> R_{kick}$ and choose $T_c=T(\widehat{R},  R_{kick},\sD)\ge1$ such that
\[
S_{t} B(\widehat{R})\subset B(\widehat{R} -  R_{kick})~~\mbox{and}~~
 S_{t-1}B(\widehat{R})\subset\sD~~\mbox{for all} ~t\ge T_c,
\]
where $\sD$ is the forward invariant absorbing set given by Theorem~\ref{th-prem1}.
This choice of $T_c$ is possible due to Remark~\ref{re:zero-force} and Theorem~\ref{th-prem1}.
\par
Now we fix $T\ge T_c$ and define a Markov chain in $W_1$
by the formula
\begin{equation}\label{chain}
U_k=F_k(U_{k-1})\equiv S_TU_{k-1}+\eta_k,~~ k=1,2,\ldots
\end{equation}
This chain is called  a random  kicks model generated  by
(\ref{fl.1a})--(\ref{fl.4a-id})
(with $G_f\equiv 0$ and $G_b\equiv 0$).
The parameter $T^{-1}$ has meaning of the frequency of the kicks.
We refer to \cite{KS-book} for  motivation and physical importance
of different types  of kick models in the turbulence theory.
\par
It is obvious that if $U_0\in B(\widehat{R})$, then  we also have
that $U_k\in B(\widehat{R})$ for all $k$.
Thus we also have a chain in  $B(\widehat{R})$.
\par
It follows from Proposition~\ref{pr:titi} that
\[
\|S_{T-1}U-S_{T-1}U_*\|_{W_1}\le C_{T,\widehat{R}}\|U-U_*\|_{W_1},~~\mbox{for all}~ U,U_*\in B(\widehat{R}).
\]
Let us fix $0<q<1$ such that $\eta=q C_{T,\widehat{R}}<1$. Since $S_{T-1} U, S_{T-1} U_*\in\sD$,
we can apply Theorem~\ref{th-main1} and choose $N$ such that
 \[
\|Q_N\big[S_{T}U-S_{T}U_*\big]\|_{W_1}\le q
\|S_{T-1}U-S_{T-1}U_*\|_{W_1}\le \eta \|U-U_*\|_{W_1}
\]
for all $U,U_*\in B(\widehat{R})$. Thus we have the squeezing property on the ball $B(\widehat{R})$.
\par
Applying now Theorem~3.2.5\cite{KS-book}\footnote{For the readers convenience we state this theorem in the Appendix.}
and the properties of the mapping $S_t$
established above
we arrive to the following assertion.

\begin{corollary}[Ergodicity]\label{co:erg}
Assume that $P_N\eta_1$ and  $Q_N\eta_1$
are independent and the distribution of $P_N\eta_1$ in the (finite-dimensional)
space $P_NW_1$ has a density $p(U)$  with respect to the Lebesgue measure
$dU$ on $P_NW_1$ such that
\[
\int_{P_NW_1}|p(U+V)-p(U)| dU\le C\|V\|_{W_1},~~~\forall V\in P_NW_1,
\]
where $C$ is a constant. Then  for the chain (\ref{chain})
 there is a unique invariant probability Borel measure $\mu$ on $B(\widehat{R})$.
 Moreover, there is $0<\gamma<1$ such that
\[
\| F_k^*\lambda-\mu\|^*_L\le C(\widehat{R},\lambda)\cdot \gamma^k, ~~ k=1,2,\ldots
\]
for any probability Borel measure $\la$ on $B(\widehat{R})$,
where the Borel measure  $F_k^*\lambda$  and the corresponding dual Lipschitz norm  $\|\cdot\|^*_L$
are defined in the  Appendix.
\end{corollary}

To  apply \cite[Theorem 3.2.5]{KS-book} (see Theorem~\ref{th:app} in the Appendix) we
first  consider
the extension $\widehat{S}_T$ of $S_T$ from  $B(\widehat{R})$
 on the whole space $W_1$ by the formula
\[
\widehat{S}_TU=\left\{ \begin{array}{ll}
S_TU, & \|U\|_{W_1}\le\widehat{R}; \\
[2mm]
S_T\left(\frac{\widehat{R}U}{\|U\|_{W_1}}\right), & \|U\|_{W_1}>\widehat{R}.
\end{array}
\right.
\]
One can see that $\widehat{S}_T$ maps $W_1$ into  $B(\widehat{R}-R_{kick})$.
Thus after the first step in \eqref{chain} we are in  $B(\widehat{R})$
and therefore we have $S_T=\widehat{S}_T$ for $k\ge 2$ in (\ref{chain}).
We also note that $\widehat{S}_T$
is globally Lipschitz (the proof of the latter property
can be found in \cite[p.64]{cl-book}, for instance)
and Conditions {\bf (A1)}--{\bf (A4)} of  Theorem~\ref{th:app}  follow directly.
\smallskip\par
For other consequences  of the Ladyzhenskaya squeezing property for the model considered we refer to 
\cite{Chu-data-as}.

 \section{Proofs}\label{sect:prf}

Our arguments are more or less standard and use the
methods developed in \cite{CaoTiti,ju,petcu} (see also the survey \cite{temam2008-srw}).
In fact  they are some refinement of the calculations known  now from
\cite{CaoTiti}, see also \cite{ju,petcu,temam2008-srw}).
This is why we consider  only key steps in the argument.
Moreover, we
 deal with the reduced system  which appears in the following way.
\par
 If we assume $G_b\equiv 0$ and $b_0=0$, we can take $b\equiv 0$
as a solution to (\ref{fl.4a}). Thus we  arrive
 to the problem: to find
 a (horizontal) fluid velocity field
\[
v(x,t)=(v^1(\bar{x},t);v^2(\bar{x},t)),~~ \bar{x}=(x,z),
\]
and the pressure $p(x,t)$ satisfying the equation
   \begin{align}\label{fl.1-main}
   v_t +(v,\g)v  -\left[\int_0^{z} \di v d\xi \right]\pd_z v -\nu [\Delta v +\pd_{zz} v]
+f v^{\perp} = -\nabla  p
 +G_f
\end{align}
 in  $\cO\times(0,+\infty)$ supplied with  conditions (\ref{fl.4a-bc1})
 and  with the initial data
    \begin{equation}\label{fl.1-main-ini}
v(x,z,0)=v_0(x,z).
  \end{equation}
It follows from Proposition~\ref{pr:titi} that for every  $G_f\in V_1$ and $v\in V_1$
problem (\ref{fl.1-main}),(\ref{fl.1-main-ini}),(\ref{fl.4a-bc1}) has  a unique strong solution
\begin{equation*}
v(t)\in C(\R_+; V_1)\cap L_2(0,T; V_2),~~~\forall T>0.
\end{equation*}
These solutions
 generate a dynamical system $(\tilde S_t,V_1)$ in $V_1$.
Below we prove our main result for this system.
Thus we deal with calculations involving the velocity field $v$, but not the buoyancy.
Moreover using approximation procedure
and the results from \cite{petcu,temam2008-srw}
we can even assume that all our calculations deal with smooth solutions.
\par
We also not that
in the case of full system (\ref{fl.1a})--(\ref{fl.4a-id})  with $G_b=0$ 
the multiplier $b$ in \eqref{fl.4a} leads to the estimate
$\|b(t)\|\le e^{-c_0 t} \|b_0\|$ for $t>0$.
Thus using dissipativity stated in Theorem~\ref{th-prem1} we can conclude that
\[
\|b(t)\|_{V_1}\le C_Be^{-c_0 t},  \quad t>t_B,
\]
for all initial data from  a bounded set $B\subset W_1$.
This means that the subspace $V_1\times \{0\}$ is exponentially attracting 
and implies that the global attractors for (\ref{fl.1a})--(\ref{fl.4a-id})
and \eqref{fl.1-main} are the same.

\subsection{Preliminary a priori estimates}\label{sect3.1}
In order to prove Theorem~\ref{th-prem1} for the system $(\tilde S_t,V_1)$
generated by (\ref{fl.1-main})
we need only to check the existence of absorbing ball in the space $V_2$.
For this we need to repeat and partially refine some calculations in \cite{ju,petcu}.
\subsubsection{A priori estimate in $H$}
If we multiply (\ref{fl.1-main}) by $v$ in $H$, then we obtain
\[
\dfrac{d\|v\|^2}{dt}+\nu\big[\|\g v\|^2+\|\pd_z v\|^2\big]\le c \|G_f\|^2,
\]
which implies that
\begin{equation*}
\|v(t)\|^2+\nu \int_t^{t+1}\big[\|\g v\|^2 +\|\pd_z v\|^2\big]d\tau \le 2e^{-c_0t}\|v(0)\|^2+c_1\|G_f\|^2
\end{equation*}
for all $t\ge 0$.
Thus there exists $R_0>0$ such that for any $R>0$ there is $t_R\ge 0$ such that
\begin{equation}\label{h-est2a}
\|v(t)\|^2+\nu \int_t^{t+1}\big[\|\g v\|^2 +\|\pd_z v\|^2\big] d\tau \le R_0^2~~~\mbox{for all}~~ t\ge t_R
\end{equation}
with $\|v(0)\|\le R$. Moreover, we can choose $R_0^2=\eps+c_1\|G_f\|^2$ with arbitrary $\eps>0$,
in this case $t_R$ also depends on $\eps$.

\subsubsection{Splitting}
Let
\[
\bar{v}=\frac{1}{L_3}\int_{-L_3/2}^{L_3/2} v dz\equiv \langle v\rangle_z~~\mbox{and}~~ \tilde{v}=v-\bar{v}.
\]
As in  \cite{CaoTiti} one can see that the fields $\bar{v}$ and $\tilde{v}$
satisfy the equations
  \begin{align}\label{fl.1-bar}
   \bar{v}_t +(\bar{v},\g)\bar{v}  +
\overline{(\tilde{v},\g)\tilde{v}} +\overline{(\g,\tilde{v})\tilde{v}}
    -\nu \Delta \bar{v}
+f \bar{v}^{\perp} = -\nabla  p
 +\overline{G_f}
\end{align}
with $\di\,\bar{v}=0$ in $\T^2\equiv (0,L_1)\times (0,L_2)$,
and
 \begin{multline}\label{fl.1-tilde}
   \tilde{v}_t +(\tilde{v},\g)\tilde{v}  -\left[\int_0^{z} \di\, \tilde{v} dz \right]\pd_z \tilde{v}
  \\
+ (\tilde{v},\g)\bar{v}+(\bar{v},\g)\tilde{v} -
[\overline{(\tilde{v},\g)\tilde{v}} +\overline{(\g,\tilde{v})\tilde{v}} ]
 \\
   -\nu [\Delta \tilde{v} +\pd_{zz} \tilde{v}]
+f \tilde{v}^{\perp} =
 \widetilde{G_f}
\end{multline}
in $\cO=\T^2\times (-L_3/2, L_3/2)$.
\subsubsection{$H^1$-estimates}

Now we multiply (\ref{fl.1-tilde}) by $|\tilde{v}|^4\tilde{v}$ and integrate over $\cO$.
In the same way as in \cite{CaoTiti,ju} we obtain
\begin{multline*}
\dfrac{d}{dt}\|\tilde{v}\|_{L_6}^6+\nu\int_\cO \left[|\g \tilde{v}|^2+
|\pd_z\tilde{v}|^2\right] |\tilde{v}|^4 d\bar{x} \\
\le C_0\left[ \|\bar{v}\|^2 \|\g \bar{v}\|^2 +\|\g\tilde{v}\|^2\right]  \|\tilde{v}\|_{L_6}^6+
C_1\|G_f\|_{L_6}  \|\tilde{v}\|_{L_6}^5.
\end{multline*}
This implies that
\begin{align*}
\dfrac{d}{dt}\|\tilde{v}\|_{L_6}^2 & \le C_0\left[1+\|v\|^2 \right] \|\g v\|^2   \|\tilde{v}\|_{L_6}^2+
C_1\|G_f\|_{L_6}  \|\tilde{v}\|_{L_6} \\
 & \le C_0\left[ 1+(1+\|v\|^2 ) \|\g v\|^2\right]   \|\tilde{v}\|_{L_6}^2+
C_1\|G_f\|^2_{L_6}.
\end{align*}
Since $\|\tilde{v}\|_{L_6}\le c\big[\|\g v\|^2+ \|\pd_z v\|^2\big]$, we can apply Lemma~\ref{gronwall} and relation
(\ref{h-est2a}) to obtain that
there exists $R_1\ge R_0>0$ such that for any $R>0$ there exists $t^*_R\ge t_R$ such that
\begin{equation}\label{tilde-v-est1}
\|\tilde{v}(t)\|_{L_6} \le R_1~~~\mbox{for all}~~ t\ge t^*_R ~~~\mbox{with}~~
\|v(0)\|\le R.
\end{equation}
Moreover using (\ref{h-est2a}) with $R_0^2=\eps+c_1\|G_f\|^2$ we can conclude that
\[
\forall \eps\in (0,1]:~~
\|\tilde{v}(t)\|^2_{L_6} \le C\left[ \eps +\|G_f\|^2\right] ~~\mbox{for all}~~ t\ge t^*_{R,\eps}
\]
with $\|v(0)\|\le R$.
Using (\ref{tilde-v-est1}) we can also assume  that
\begin{equation}\label{tilde-v-est2}
\int_t^{t+1}\int_\cO \left[|\g \tilde{v}|^2+
|\pd_z\tilde{v}|^2\right] |\tilde{v}|^4 d\bar{x}d\tau  \le C(R_1)~~~\mbox{for all}~~ t\ge t^*_R
\end{equation}
with $\|v(0)\|\le R$.
In the case $G_f\equiv 0$ we can change  $C(R_1)$ into $C\eps$ and $t^*_R$ into $t^*_{R,\eps}$.

\par
Next we multiply (\ref{fl.1-bar}) by $-\Delta \bar{v}$ in $L_2(\T^2)$. As in \cite{CaoTiti}
using the relations
\[
\int_{\T^2} (\bar{v},\g)\bar{v} \Delta \bar{v} dx=0,~~
\int_{\T^2}\g p \Delta \bar{v}dx=0, ~~
\int_{\T^2}\bar{v}^{\perp} \Delta \bar{v}dx=0,
\]
 we have that
 \[
 \hf \frac{d}{dt} \|\g\bar{v}\|^2 +
\nu  \|\Delta \bar{v}\|^2
=
 \int_{\T^2}\left[
\overline{(\tilde{v},\g)\tilde{v}} +\overline{(\g,\tilde{v})\tilde{v}}\right] \Delta \bar{v} dx
 - \int_{\T^2}\overline{G_f}\Delta \bar{v} dx.
\]
This  implies (see \cite{CaoTiti} for some details) that
 \begin{align*}
  \frac{d}{dt} \|\g\bar{v}\|^2 +
\nu  \|\Delta \bar{v}\|^2
\le & C\Big[  \|\g \tilde{v}\|^2  +
 \int_{\cO}|\tilde{v}|^4|\g \tilde{v}|^2 d\bar{x}
+\|{G_f}\|^2 \Big].
\end{align*}
Therefore (\ref{tilde-v-est1}) and (\ref{tilde-v-est2}) give us
that
 there exists $R_{2}>0$ such that for any $R>0$ there is $t^{**}_R\ge t^{*}_R$ such that
\begin{equation}\label{bsr-v-est}
\|\g\bar{v}(t)\|^2+\nu \int_t^{t+1}\|\Delta \bar{v}\|^2 d\tau \le R_2^2~~~\mbox{for all}~~ t\ge t^{**}_R ~~~\mbox{with}~~
\|v(0)\|\le R.
\end{equation}
As above in the case $G_f\equiv 0$ we can take $R^2_2$ of the order $\eps$ (with $t^{**}_R$ depending on
$\eps$).
\par
The next step is the estimate for $v_z$. We first note that $u\equiv v_z$ solves the problem
  \begin{align}\label{vz-eq}
   u_t +(v,\g)u  -\left[\int_0^{z} \di v d\xi \right]\pd_z u
 &  -\nu [\Delta u +\pd_{zz}u] \notag
\\
= & - f u^{\perp}  -(u,\g)v +(\g,v) u+
 \pd_zG_f
\end{align}
in the class of periodic (\textit{odd} in $z$) functions.
Since $u^{\perp} u =0$ and
\[
 \int_{\cO}\left\{(v,\g)u  -\left[\int_0^{z} \di v d\xi \right]\pd_z u\right\} u d\bar{x}=0,
\]
using the multiplier $u$ after integration by parts we obtain that
\[
\hf\dfrac{d\|u\|^2}{dt}+\nu\big[\|\g u\|^2+\|\pd_z u\|^2\big]\le \|\pd_zG_f\| \|u\|
+ \int_\cO |v| |u| |\g u| d\bar{x},
\]
which implies (for some details we refer to \cite{CaoTiti}) that
  \begin{align*}
\dfrac{d\|u\|^2}{dt}+\nu\big[\|\g u\|^2+\|\pd_z u\|^2\big]
&
\le C\left[ \|v\|^4_{L_6}\|u\|^2+ \|\pd_zG_f\|^2 \right] \\
& \le C\left[ \|\tilde{v}\|^4_{L_6} + \|\g \bar{v}\|^4\right] \|u\|^2+ C\|\pd_zG_f\|^2.
\end{align*}
Therefore
 there exists $R_{3}>0$ such that for any $R>0$ there is $t^{(3)}_R\ge t^{**}_R$ such that
\begin{equation*}
\| v_z(t)\|^2+\nu \int_t^{t+1}\left[\|\g v_z\|^2+\|v_{zz}\|^2\right] d\tau \le R_3^2~~~\mbox{for all}~~ t\ge t^{(3)}_R
\end{equation*}
with  $\|v(0)\|\le R$.
Moreover,  $R^2_2\sim \eps$ in the case $G_f\equiv 0$.
\par
Next we estimate $\|\g v\|$. For this we multiply (\ref{fl.1-main}) by $-\Delta v$.
As in \cite{CaoTiti} we obtain
  \begin{multline*}
\dfrac{d\|\g v\|^2}{dt}+\nu\big[\|\Delta v\|^2+\|\g \pd_z v\|^2\big]
\\
\le C\left[ \|v\|^4_{L_6}+ \|\g v\|^2\| v_z\|^2 \right]\|\g v\|^2+ C\|G_f\|^2
\end{multline*}
which implies again that
there exists $R_{4}>0$ such that for any $R>0$ there is $t^{(4)}_R\ge t^{(3)}_R$ such that
\begin{equation}\label{vz-est2}
\| \g v(t)\|^2+\nu \int_t^{t+1}\left[\|\Delta v \|^2+\|\g v_{z}\|^2+\|v_{zz}\|^2\right] d\tau \le R_4^2
\end{equation}
for all $t\ge t^{(4)}_R$ with  $\|v(0)\|\le R$.
Moreover one can see from the analysis above that in the case $G_f\equiv 0$ for any $0<\eps<1$
there exists $t_{R,\eps}>0$ such that
\begin{equation}\label{vz-est2-eps}
\| \g v(t)\|^2+\nu \int_t^{t+1}\left[\|\Delta v \|^2+\|\g v_{z}\|^2+\|v_{zz}\|^2\right] d\tau \le \eps~~~\mbox{for all}~~ t\ge t_{R,\eps}
\end{equation}
provided  $\|v(0)\|\le R$.
\par
The relation in (\ref{vz-est2}) yields that the system $(\tilde{S}_t,V_1)$
is dissipative in $V_1$.
In the case $G_f\equiv 0$ by \eqref{vz-est2-eps} an absorbing
 ball can be chosen with an arbitrary small radius
which implies that the zero equilibrium is asymptotically stable in this case
(as  it is claimed in  Remark~\ref{re:zero-force} when the buoyancy is presented).

\subsection{Completion of the proof of Theorem~\ref{th-prem1}}
To conclude the proof we need to show uniform smoothing of trajectories.
We note that the result for individual trajectories is known from
\cite{petcu} (see also \cite{temam2008-srw}). However in these references
the size of the corresponding ball is not controlled and thus we need some refinement
of that results. This is why below we mainly follow the line of argument in \cite{petcu,temam2008-srw}.
 \par
We multiply (\ref{fl.1-main}) by $(-\Delta_{x,z})^2v$
to get
\begin{align}\label{Delta2-est}
 \hf \frac{d}{dt} \|\Delta_{x,z}v \|^2 +
\nu  \|[-\Delta_{x,z}]^{3/2}v\|^2
= &
 - \int_{\cO}\left[ (v,\g)v  +w(v)\pd_z v\right](-\Delta_{x,z})^2v
 d\bar{x} \notag \\ & + (G_f,(-\Delta_{x,z})^2v),
\end{align}
where $w(v)=w(\bar{x},t)$ is defined  by \eqref{rep-w}.
We also use the facts that
\[
\int_{\cO} v^{\perp} (-\Delta_{x,z})^2vd\bar{x}=\int_{\cO} [\Delta_{x,z}v]^{\perp} \Delta_{x,z}vd\bar{x}=0
\]
and
\[
\int_{\cO} \g p  (-\Delta_{x,z})^2vd\bar{x}=0.
 \]
The general structure of the nonlinear terms in (\ref{Delta2-est}) is given by the formula
\begin{align*}
N(v) \equiv & \int_{\cO}\left[ (v,\g)v  +w(v)\pd_z v\right](-\Delta_{x,z})^2v
 d\bar{x} \notag \\
= & \int_{\cO} \Delta_{x,z}\left[ (v,\g)v  +w(v)\pd_z v\right] \Delta_{x,z}v d\bar{x}
\notag \\
= & \int_{\cO}\left[ (v,\g) \Delta_{x,z}v  +w(v)\pd_z \Delta_{x,z} v\right] \Delta_{x,z}v d\bar{x}
+\Sigma(v) = \Sigma(v).
\end{align*}
Here
\begin{align}\label{Delta2-nl-terms2}
\Sigma(v)
 & \sim\sum_{k,i,j} \int_{\cO}\left[ D^2 v^i D^{1} v^j D^2v^j
 +D^kw(v)D^{2-k}\pd_zv^j \Delta_{x,z}v^j \right]d\bar{x},
\end{align}
where $k,i,j=1,2$, $D^k$ is a differential operator of order $k$, and
the sign $\sim$ means that we deal with linear combination of the terms in the
right hand side.
\par
Using the embeddings $H^1(\cO)\subset L_6(\cO)$  and  $H^{1/2}(\cO)\subset L_3(\cO)$
(see \cite{Triebel78}) and also interpolation
one can see that
\begin{align*}
\Sigma_1 &\equiv\left|\int_{\cO} D^2 v^i D^{1} v^j \Delta_{x,z}v^j
d\bar{x}\right| \le C \|  D^2 v_i\| \|D^{1} v_j\|_{L_6} \|D^2v_j\|_{L_3} \\
&\le C \|v\|_{H^2}^2 \|D^2v^j\|^{1/2} \|D^2v^j\|^{1/2}_{H^1}
\le \eps \|[-\Delta_{x,z}]^{3/2}v\|^2 +C_\eps \|v\|_{H^2}^{10/3}
\end{align*}
To estimate the second term in (\ref{Delta2-nl-terms2}) we start with the case $k=2$:
\begin{align*}
\int_{\cO}D^2w(v)\pd_zv^j \Delta_{x,z}v^j d\bar{x} &\le
C \|D^2w(v)\| \|\pd_z v^j\|_{L_6} \|D^2v^j\|_{L_3} \\
&\le C \|v \|_{H^3}^{3/2} \|\pd_z v\|_{L_6} \|v\|^{1/2}_{H^2} \\
& \le \eps \|[-\Delta_{x,z}]^{3/2}v\|^2 +C_\eps \|\pd_z v\|_{L_6}^4\|v\|_{H^2}^{2}.
\end{align*}
To estimate the term with $k=1$ in (\ref{Delta2-nl-terms2}) we use the following lemma
(see Proposition~2.2 in \cite{CaoTiti03}).
\begin{lemma}\label{le:titi}
For smooth functions $\phi$ and $\psi$ and vector field $v=(v^1;v^2)$ we have
\[
\left|\int_\cO\left[\int_0^{z} \di v(x,\xi) d\xi\right]\psi \phi d\bar{x}\right|
\le C \|v\|_{H^2}^{1/2} \|v\|_{H^1}^{1/2} \|\psi\|_{H^1}^{1/2} \|\psi\|^{1/2} \|\phi\|.
\]
\end{lemma}
This lemma yields the following estimate
\begin{align*}
\int_{\cO}D^1w(v) D^1\pd_zv^j \Delta_{x,z}v^j d\bar{x} &\le
C \|v\|_{H^2}^{1/2} \| v\|_{H^3}^{1/2}\|\pd_z v^j\|_{H^2}^{1/2} \|\pd_z v^j\|_{H^1}^{1/2} \|D^2v\| \\
&\le C \|v \|_{H^3} \|v\|^{2}_{H^2}
 \le \eps \|[-\Delta_{x,z}]^{3/2}v\|^2 +C_\eps \|v\|_{H^2}^{4}.
\end{align*}
Thus we arrive at the following relation
\begin{align}\label{Delta2-est-f}
 \frac{d}{dt} \|\Delta_{x,z}v \|^2 +
\nu  \|[-\Delta_{x,z}]^{3/2}v\|^2
\le &
 C_1 \left[ 1+  \|\Delta_{x,z}v \|^2  +\|\pd_z v\|^4_{L_6}\right]  \|\Delta_{x,z}v \|^2
 \notag \\ & + C_2 \|G_f\|_{V_1}^2.
\end{align}
Our next step is an estimate for $\|\pd_z v\|_{L_6}$.
As above let $u=\pd_zv$.
We multiply (\ref{vz-eq}) by $|u|^4u$. Using
the relations
 $u^{\perp} u =0$ and
\[
 \int_{\cO}\left\{(v,\g)u  -\left[\int_0^{z} \di v d\xi \right]\pd_z u\right\} |u|^4u d\bar{x}=0,
\]
we obtain that
\begin{multline*}
\frac16\dfrac{d}{dt}\|u\|^6_{L_6}+\nu\int_\cO\big[ |\g u|^2+|\pd_z u|^2\big]
|u|^4 d\bar{x} \\
\le \int_{\cO}\left\{  - (u,\g)v  +(\g,v) u\right\} |u|^4u d\bar{x}
 +(\pd_zG_f, |u|^4u).
\end{multline*}
In the same way as in \cite{petcu,temam2008-srw} after integration by parts
we have that
\begin{align*}
  \int_{\cO}\left\{  - (u,\g)v  +(\g,v) u\right\} |u|^4u d\bar{x} &\le
  C\|v\|_{L_\infty}  \int_{\cO}|\g u|  |u|^5 d\bar{x} \\
& \le   C\|v\|_{L_\infty}  \left[ \int_{\cO} |\g u |^2  |u|^4 d\bar{x} \right]^{1/2} \| u\|^3_{L_6}
\\ &
\le \eps \int_{\cO} |\g u |^2  |u|^4 d\bar{x} + C_\eps \|v\|_{L_\infty}^2 \| u\|^6_{L_6}.
\end{align*}
This implies that
\begin{multline*}
\dfrac{d}{dt}\|u\|^6_{L_6}+\nu\int_\cO\big[ |\g u|^2+|\pd_z u|^2\big]
|u|^4 d\bar{x} \\
\le C\left[ \|v\|^{2}_{H^2} \|u\|^6_{L_6} +
 \|\pd_zG_f\|_{L_6}\|u\|^5_{L_6}\right].
\end{multline*}
Thus we have that
\[
\dfrac{d}{dt}\|u\|^2_{L_6}
\le C_1\left[1+ \|\Delta_{x,z}v \|^2\right] \|u\|^2_{L_6} +
 C_2\|\pd_zG_f\|_{L_6}^2.
\]
Using Lemma~\ref{gronwall} and the estimate in \eqref{vz-est2} we obtain that
\[
\|\pd_z{v}(t)\|_{L_6} \le R_5~~~\mbox{for all}~~ t\ge t^{(5)}_R\ge t^{(4)}_R ~~~\mbox{with}~~
\|v(0)\|\le R.
\]
Now returning to \eqref{Delta2-est-f} and applying
 Lemma~\ref{gronwall} allow us to obtain the following assertion.
\begin{proposition}\label{pr:dis}
Let $G_f\in V_1$ and $\pd_zG_f\in L_6(\cO)$. Then
there exists $R_*>0$ such that the ball $\sB=\{ v\in V_2: \|v\|_{V_2}\le R_*\}$ is absorbing
for the dynamical system $(\tilde S_t,V_1)$
generated by problem (\ref{fl.1-main}).
Moreover,  there is a forward invariant  absorbing set $\sD$ in this ball.
\end{proposition}
Thus the proof of Theorem~\ref{th-prem1}
(in the case $b_0\equiv 0$ and $G_b\equiv 0$)
is complete.

\subsection{Squeezing estimate}
Now we prove Theorem~\ref{th-main1} (in the case $b_0\equiv 0$ and $G_b\equiv 0$).
We start with the following assertion.
\begin{proposition}\label{pr:lip}
For any $v_*,v_{**}$ from the absorbing forward invariant set $\sD$ we have
the relation
\[
\|\tilde S_tv_{*}-\tilde S_tv_{**}\|_{V_1}\le C_{\sD}e^{\alpha_{\sD}t}\|v_{*}-v_{**}\|_{V_1},~~~\forall\, t\ge 0,
\]
where $C_{\sD}$ and $\alpha_{\sD}$ are positive constants.
Moreover, for $v_*(t)=\tilde S_tv_*$ and $v_{**}(t)=\tilde S_tv_{**}$
we also have that
\begin{equation*}
\|u(t)\|_{V_1}^2+\int_0^t\|u(\tau)\|_{V_2}^2d\tau\le C_{\sD}e^{\alpha_{\sD}t}\|v_*-v_{**}\|^2_{V_1},
~~~\forall\, t\ge 0,
\end{equation*}
where $u(t)=v_*(t)-v_{**}(t)$.
\end{proposition}
\begin{proof}
We use the same type of argument as in \cite{ju}. For this we write the equation
  \begin{multline}\label{fl.1-main-dif}
   u_t   -\nu [\Delta u +\pd_{zz} u]+(u,\g)v_{*}  + (v_{*},\g)u
\\
    -\left[\int_0^{z} \di u d\xi \right]\pd_z v_{*}
   -\left[\int_0^{z} \di v_{**} d\xi \right]\pd_z u
+f u^{\perp} = -\nabla  (p_*-p_{**})
 \end{multline}
 for the difference $u(t)=v_*(t)-v_{**}(t)=S_tv_*-S_tv_{**}$.
 Then
  we multiply \eqref{fl.1-main-dif} by $(-\Delta_{x,z})u$ and apply the same argument as in \cite{ju}.
 At the final stage we use the fact that
\[
\| v_*(t)\|_{V_2}+\|v_{**}(t)\|_{V_2}\le C_\sD,~~~t>0,
 \]
 for initial data $v_*$ and $v_{**}$ from $\sD$.
\end{proof}
\par
Let $\{\la_k\}$ be the eigenvalues of the corresponding Stokes type operator $A$
which arises in problem \eqref{fl.1-main}
and $\{e_k\}$ be the corresponding eigenvectors. We can assume that
\[
0<\la_1\le\la_2\le \ldots;~~~\lim_{k\to\infty}\la_k=\infty.
\]
We denote by $P_N$ the orthoprojector in $H$ on Span$\{e_1,\ldots,e_N\}$.
Let $Q_N=I-P_N$. We multiply (\ref{fl.1-main-dif}) by $Q_NAu$ and obtain that

\begin{align}\label{N-est}
 \hf \frac{d}{dt} \|A^{1/2} Q_Nu \|^2 + &
\nu  \|A Q_Nu\|^2
=
-([(u,\g)v_{*}  + (v_{**},\g)u],  Q_NAu)
\notag \\ &
    + \int_{\cO}\left[\int_0^{z} \di u d\xi \right]\pd_z v_{*} Q_NAud\bar{x}-f (u^{\perp},Q_NAu)
 \notag \\ &  +\int_{\cO} \left[\int_0^{z} \di v_{**} d\xi \right]\pd_z u Q_NAud\bar{x}.
\end{align}
Now we need the following lemma which follows from the
H\"older inequality and from the standard Sobolev embedding (see, e.g., \cite{Triebel78}):
\[
H^s(\cO)\subset L_p(\cO)~~~\mbox{when}~~~ s-3/2= -3/p,~~~p\ge 2,~~\cO\subset\R^3.
\]
\begin{lemma}\label{le:tambaca}
Let $\cO\subset\R^3$ and $u$, $v$ be sufficiently smooth functions.
Then
\[
\|uv\|_{L_2}\le C \|u\|_{H^s} \|v\|_{H^r}
\]
for any $0<s,r<3/2$. such that $r+s=3/2$.
\end{lemma}
Using this lemma to estimate the right hand side in (\ref{N-est})
on the forward invariant set $\sD$
 we can find that
\begin{align*}
 \hf \frac{d}{dt} \|A^{1/2} Q_Nu \|^2 + &
\nu  \|A Q_Nu\|^2
\le C_{R_*}\| u\|_{H^{3/2}} \|A Q_Nu\|,
\end{align*}
where $R_*$ is the same as in the statement of Proposition~\ref{pr:dis}.
This yields
\begin{align*}
 \frac{d}{dt} \|A^{1/2} Q_Nu \|^2 + &
\nu  \|A Q_Nu\|^2
\le C_{R_*,\eps}\| u\|^2_{V_1} + \eps\| u\|^2_{V_2}
\end{align*}
for any $\eps>0$. This implies that
\begin{multline*}
 \|A^{1/2} Q_Nu(t) \|^2 \le  e^{-\nu\la_{N+1}t} \|A^{1/2} Q_Nu(0) \|^2
\\ +C_{R_*,\eps}\int_0^te^{-\nu\la_{N+1}(t-\tau)}
 \| u\|^2_{V_1}d\tau + \eps\int_0^te^{-\nu\la_{N+1}(t-\tau)}\| u\|^2_{V_2} d\tau.
\end{multline*}
Using Proposition~\ref{pr:lip} we obtain that
\begin{align*}
 \|A^{1/2} Q_Nu(t) \|^2 \le & \Big[ e^{-\nu\la_{N+1}t} +\left(\frac{C_{R_*,\eps}}{\al_{R_*}+\nu\la_{N+1}}
 + \eps C_{R_*}\right) e^{\al_{R_*} t}\Big]  \|A^{1/2}u(0) \|^2.
\end{align*}
Now for any fixed $T>0$ and  $0<q<1$ we can choose $\eps$ and $N$ such that
\begin{align*}
 \|A^{1/2} Q_Nu(T) \|^2 \le  q^2  \|A^{1/2}u(0) \|^2.
\end{align*}
Thus we arrive to the squeezing property:
For every $T>0$ and  $0<q<1$ there is $N$ such that
\[
\|Q_N[\tilde  S_Tv_1-\tilde  S_Tv_2]\|_{V_1}\le q \|v_1-v_2\|_{V_1}
\]
for any $v_1,v_2$ from the absorbing forward invariant set $\sD$.

\appendix
\section{Appendix: Kuksin-Shirikyan ergodicity  \\ theorem}
In this Appendix for the readers convenience we state   Theorem~3.2.5~\cite{KS-book}
on the ergodicity of a class of Markov chains.
\par
Let $H$ be a separable Banach space. We denote by  $B(R)$  the ball with the center at $0$ of the radius $R$
in the space $H$. Let $S$ be a   continuous (nonlinear) operator from $H$ into itself.
Let $\{\eta_k\}$ be i.i.d.\ random variables in $H$
such that $\|\eta_k\|\le r_*$ for some $r_*>0$.
We consider a Markov chain  in $H$ defined by the relations
\begin{equation}\label{chain-app}
U_k=F_k(U_{k-1})\equiv SU_{k-1}+\eta_k,~~ k=1,2,\ldots
\end{equation}
We assume that the following  hypotheses  are satisfied.
\begin{enumerate}
 \item[{\bf(A1)}]  $S$ is a Lipschitz mapping on every ball  $B(\varrho)$ and for any $R>r>0$
 there exists $n_*=n(R,r)$ such that   $S^nB(R)\subset B(r)$ for all $n\ge n_*$.
  \item[{\bf(A2)}]
 There exists  $\widehat{R}> 0$  such that for every $R>0$ we have that
 $\cA_k(R)\subset   B(\widehat{R})$ for all  $k\ge k_*(R)$, where the sets $\cA_k(R)$ is defined by the relations
 \[
  \cA_0(R)=B(R),~~~   \cA_k(R)=S (\cA_{k-1}(R))+ B(r_*),~~ k=1,2,\ldots
 \]
 \item[{\bf(A3)}]
 There exist a finite dimensional projector $P$ in $H$
and  a number $\eta<1$ such that
 \[
\|(I-P)\big[S U-S U_*\big]\|_{H}\le  \eta \|U-U_*\|_{H},~~\forall\,
U,U_*\in {\rm Cl}\big[ \cup_{k\ge 0}\cA_k(\widehat{R})\big].
\]
\item[{\bf(A4)}]
 Assume that $P\eta_1$ and  $(I-P)\eta_1$
are independent and the distribution of $P\eta_1$ in the (finite-dimensional)
space $P H$ has a density $p(U)$  with respect to the Lebesgue measure
$dU$ on $PH$ such that
\[
\int_{PH}|p(U+V)-p(U)| dU\le C\|V\|_{H},~~~\forall V\in PH,
\]
where $C$ is a constant.
 \end{enumerate}
We also recall (see, e.g., \cite{KS-book}) that
 the  dual Lipschitz norm  in the space of probability Borel measures  on $H$ is defined by the relation
\[
 \|\mu_1-\mu_2\|^*_L=\sup\left\{ \left| \int_Hf(v)\mu_1(dv)-\int_Hf(v)\mu_2(dv)\right| :
\begin{array}{l}
   f\in L_b(H), \\ \|f\|_L\le1
  \end{array}\right\},
\]
where $L_b(H)$ is the space of bounded  Lipschitz functions on $H$
endowed with the norm
\[
 \|f\|_L= \sup_{v\in H}|f(v)|+\sup_{v_1,v_2\in H}\frac{|f(v_1)-f(v_2)|}{\|v_1-v_2\|_H}.
\]
Now we are in position to state  Theorem~3.2.5\cite{KS-book}.

\begin{theorem}[Ergodicity; Kuksin-Shirikyan \cite{KS-book}]\label{th:app}
If we assume that hy\-po\-theses {\bf(A1)}--{\bf(A4)} are in force,
 then  for the chain (\ref{chain-app})
 there is a unique invariant probability Borel measure $\mu$ on $H$.
 Moreover, there is $0<\gamma<1$ such that
\[
\| F_k^*\lambda-\mu\|^*_L\le C e^{-\al k}\int \exp\{\al k_*(\|U\|)\}\la (dU),
 ~~ k=1,2,\ldots
\]
for any probability Borel measure $\la$ on $H$ for which the integral in the right-hand side
of the relation above is finite (the function $k_*(R)$ is defined in {\bf(A2)}).
Here above  the probability  measure  $F_k^*\lambda$  is defined
on a Borel set $\Gamma$ by the formula
\[
(F_k^*\lambda)(\Gamma)=\int_H P_k(V,\Gamma) \la(dV),
\]
where $P_k(U_0,\Gamma)=\p (U_k\in\Gamma)$ is the transition probability for the chain.
\end{theorem}

We note that the result in Theorem~\ref{th:app} formally deals with
a Markov chain in a {\em whole} space.
However one can see that the arguments from  \cite{KS-book}
remain true if we restrict the iteration procedure
in (\ref{chain})  on an invariant subset. This observation was already used
in  \cite{CK-arma}  to prove ergodicity
for random kick-forced 3D Navier-Stokes equations in a thin domain and can be
also applied to obtain Corollary \ref{co:erg}.

\subsection*{Acknowledgement}
The paper  was written as an answer to the question posed
by Sergei Kuksin concerning a squeezing property  for
viscous 3D primitive equations. The author
would like to express his gratitude to him and also to Armen Shirikyan
for fruitful discussions.
\par
This research was completed during my 6 weeks visit to the Institute for
Mathematics and its Applications (Minneapolis) in the framework of
2012/2013 Annual Program ``Infinite Dimensional and Stochastic Dynamical Systems
and their Applications" and was supported in part by the Institute
with funds provided by the NSF.

\end{document}